\documentclass[12pt, reqno]{amsart}
\usepackage{amsmath,amssymb,amsfonts,amstext, amsthm, amscd}
\usepackage{verbatim}
\usepackage{graphicx}
\usepackage{color}
\usepackage{enumerate}
\usepackage{epstopdf}

\newtheorem{theorem}{Theorem}
\newtheorem{lemma}[theorem]{Lemma}
\newtheorem{proposition}[theorem]{Proposition}
\newtheorem{corollary}[theorem]{Corollary}

\newtheorem*{theorem-nonum}{Theorem 3}
\newtheorem{fact}{Fact}

\theoremstyle{definition}
\newtheorem*{definition*}{Definition}
\newtheorem{remark}[theorem]{Remark}

\newcommand{\Z}{{\mathbb Z}}
\newcommand{\R}{{\mathbb R}}
\newcommand{\N}{{\mathbb N}}

\newcommand\B{{\mathcal B}}        

\newcommand\garbage[1]{}

\renewcommand{\P}{{\mathbb P}}

\usepackage{bbm}

\newcommand{\aS}{S\alpha S}

\newtheorem{manualtheoreminner}{Theorem}
\newenvironment{manualtheorem}[1]{%
	\renewcommand\themanualtheoreminner{#1}%
	\manualtheoreminner
}{\endmanualtheoreminner}

\begin{document}
	\title{Stable CLT for deterministic systems}
	
\author{Zemer Kosloff}
\thanks{The research of Z.K. was partially supported by ISF grant No. 1570/17}
\address{Einstein Institute of Mathematics,
	Hebrew University of Jerusalem, Edmond J. Safra Campus, Jerusalem 91904,
	Israel}
\email{zemer.kosloff@mail.huji.ac.il}

\author{Dalibor Voln\'y}
\address{Laboratoire de Math\'{e}matiques Raphael Salem,
	UMR 6085, Universit\'e de Rouen Normandie, France}
\email{dalibor.volny@univ-rouen.fr}


	
\maketitle

\begin{abstract}
We show that for every ergodic and aperiodic probability preserving transformation and $\alpha\in (0,2)$ there exists a function whose associated time series is in the standard domain of attraction of a non-degenerate symmetric $\alpha$-stable distribution.
	\end{abstract}

\section{Introduction}
A random variable $Y$ is \textbf{stable} if there exists a sequence $Z_1,Z_2,\ldots$ of i.i.d. random variables and sequences $a_n,b_n$ such that
\[
\frac{\sum_{k=1}^n Z_k-a_n}{b_n},\ \text{converges in distribution to} \ Y,\ \text{as}\ n\to\infty. 
\] 
In other words, $Y$ arises as a distributional limit of a central limit theorem, see \cite{MR0322926}. Furthermore in this case, $b_n$ is regularly varying of index $\frac{1}{\alpha}$ which implies that $b_n=n^{1/\alpha}L(n)$ where $L(n)$ is a slowly varying function.  The Normal and the Cauchy distribution are stable distributions and one can parametrize the class of stable distribution via their characteristic functions (Fourier transform). Namely a random variable is $\alpha$-stable, $0<\alpha\leq 2$, if there exists $\sigma>0$, $\beta\in[-1,1]$ and $\mu\in\R$ such that for all $\theta\in\R$. 
\[
\mathbb{E}(\exp(i\theta Y))=\begin{cases}
\exp\left(-\sigma^\alpha|\theta|^\alpha(1-i\beta(\text{sign}(\theta)\tan(\frac{\pi\alpha}{2})+i\mu\theta)\right), &\alpha\neq 1,\\
\exp\left(-\sigma^\alpha|\theta|^\alpha(1+\frac{i\beta}{2}(\text{sign}(\theta)\ln (\theta)+i\mu\theta)\right), & \alpha=1.
\end{cases}
\]
The constant $\sigma>0$ is the dispersion parameter and $\beta$ is the skewness parameter. In this case we will say that $Y$ is a $S_\alpha(\sigma,\beta,\mu)$ random variable.  If $\mu=\beta=0$ and $\sigma>0$ then the random variable is \textbf{symmetric $\alpha$ stable} and we will abbreviate $Y$ is $S\alpha S(\sigma)$. 
See \cite{MR1280932} for a detailed account of infinite variance ($\alpha\neq 2$) stable processes and its appearance in various fields of mathematics and science. 

A \textbf{probability preserving dynamical system} is a quadruplet $(\mathcal{X},\B,m,T)$ where $(\mathcal{X},\B,m)$ is a standard probability space and  $T:\mathcal{X}\to \mathcal{X}$ is measurable and $m\circ T^{-1}=m$. The system is \textbf{aperiodic} if the collection of all periodic points is a null set.  It is \textbf{ergodic} if every $T$-invariant set is either a null or a co-null set.

A function $f:\mathcal{X}\to \R$ generates a stationary process $(f\circ T^n)_{n=1}^\infty$ and $S_n(f)=\sum_{k=0}^{n-1}f\circ T^k$ is its corresponding \textbf{sum process}. Given $Y$ a $S\alpha S(\sigma)$ random variable, a function $f$ is a \textbf{$Y$-CLT function} if there exists $b_n\to\infty$ and $a_n$ such that $\frac{S_n(f)-a_n}{b_n}$ converges in distribution to $Y$. Since the distribution of $Y$ is non-atomic, this is equivalent to: for all $t\in\R$, 
\[
\lim_{n\to\infty}m\left(\frac{S_n(f)-a_n}{b_n}\leq t\right)=\P(Y\leq t)
\] 
If in addition $a_n=0$ and $b_n=n^{1/\alpha}$ then the time series generated by $f$ is in the \textbf{standard domain of attraction} of $Y$. 

It seems that general methods of proof of the central limit theorem in the dynamical systems setting work only in the case of positive entropy systems. For example, if $\left(f\circ T^n\right)_{n=0}^\infty$ is a martingale difference sequence and $T$ has zero entropy then $f\equiv 0$. Consequently martingale approximation can hardly be used. It was a natural open problem whether  every aperiodic dynamical system admits a function which satisfies the CLT with a nondegenerate normal distribution as a limit.   

In 1986, Burton and Denker \cite{MR891642} answered this question in the affirmative by showing that for every aperiodic dynamical system, even very deterministic ones such as irrational rotations, there exists a CLT function $f$ for $Y$, a standard normal distribution.

For the moment suppose that $Y$ is a standard normal random variable. By $L_0^2$ we denote the space
of $L^2$ functions with zero mean. One can notice that the functions found by Burton and Denker are from $L_0^2$.
As remarked in \cite{MR891642}, because coboundaries are dense in $L_0^2$, the set of Y-CLT functions is dense in $L_0^2$.
As shown in \cite{MR1093661} for any sequence $b_n\to\infty$, $b_n=o(n)$, there exists a dense $G_\delta$ subset of $f\in L_0^2$
such that every probability law is a weak limit of the distributions of $(1/b_n) S_n(f)$. The set of Y-CLT functions
is therefore meagre.

 Burton and Denker asked whether there exists a function $f$ which satisfies the Weak Invariance Principle (WIP), meaning that the partial sums process $W_n:X\times [0,1]$, $W_n(t)=\frac{1}{\sqrt{n}}\sum_{k=0}^{[nt]-1}f\circ T^k$, when viewed as a random process with values in the space of C\`adl\`ag functions converges in distribution to a Brownian motion. This question was resolved in the affirmative by the second author in \cite{MR1624218} and recently we showed in \cite{MR4374685} that when $T$ is ergodic and aperiodic there exists a function $f:X\to\Z$ for which the lattice local central limit theorem holds. 

Weiss and Thouvenot showed in \cite{MR2887924} that for every free probability preserving system and random variable $Y$ there exists a function $f$ such that $\frac{1}{n}S_n(f)$ converges in distribution to $Y$. See also \cite{MR3795070} where a refined result for positive valued processes is obtained with normalizing constants of the form $\frac{1}{c_n}$ with $c_n$ a $1$-regularly varying sequence.  

In this work we show the existence of CLT functions for the whole range of symmetric $\alpha$ stable distributions with the scaling $b_n=n^{1/\alpha}$. This normalization corresponds to that for iid sequences, unlike the others mentioned above.   
\begin{theorem}\label{thm: CLT}
	Let $(\mathcal{X},\B,m,T)$ be an ergodic, aperiodic probability preserving system, For every $\alpha \in (0,2)$, $\sigma>0$, there exists $f:\mathcal{X}\to\R$  such that
	$\frac{1}{n^{1/\alpha}}S_n(f)$ converges in distribution to a $S\alpha S(\sigma)$ random variable. 
\end{theorem}

We remark that a considerable part of the statement is that the scaling is of the form $n^{1/\alpha}$. One reason for interest in this scaling is that if a stationary process satisfies a WIP with a non-degenerate $S\alpha S$ L\'evy motion as a limit then $b_n$ must be $1/\alpha$ regularly varying. Furthermore, by Fact \ref{Fact: Lesigne}, when $\alpha\in (0,1)$, this scaling is the largest possible growth rate of the dispersion parameter for the sum process of a stationary $\aS$ process.  

\subsection{Organisation of the paper} 

In Section \ref{sec: Stable} we introduce a carefully chosen triangular array and use a Proposition 2 from \cite{MR4374685} to embed it in a given aperiodic, ergodic probability preserving system. We then construct, using the functions from the embedding, the function which satisfies the $\alpha$-stable CLT. 

Section \ref{sec: proofs} is concerned with the proof of the CLT for the function from Section \ref{sec: Stable}. The last section is a short appendix containing some standard properties of $\aS$ random variables which are used in Section \ref{sec: proofs}. 




\subsubsection{Notations}
In what follows we will write for $f,g$ two positive valued functions (or sequences), $f(t)\sim g(t)$ if $\lim_{t\to\infty}\frac{f(t)}{g(t)}=1$. We will denote by $f(t)\lesssim g(t)$ if there exists $C>0$ such that $f(t)\leq Cg(t)$ for all large $t$ and $f\approx g$ if $f(t)\lesssim g(t)$ and $g(t)\lesssim f(t)$. 

In addition when $f$ and $g$ are real valued functions with $2\leq f(t)\leq g(t)$, we write $\sum_{k=f(t)}^{g(t)}a_k$ for the sum $\sum_{k=[f(t)]}^{[g(t)]}a_k$ where $[x]$ is the floor function of $x$. 

Given a sequence $(Y_n)_{n=1}^\infty$ of random variables and a random variable $Y$, $Y_n\Rightarrow^d Y$ denotes \textit{$Y_n$ converges in distribution to $Y$}, $X=^d Z$ means \textit{$X$ and $Y$ are equally distributed} and $Y\sim^d \aS(\sigma)$ means \textit{$Y$ is distributed  $\aS(\sigma)$}.

For a sequence of random variables $Y(1),Y(2),\ldots$ and $n\in \N$, we write  $S_n(Y)=\sum_{j=1}^nY(j)$. 



\section{Stable laws and a CLT for a target process}\label{sec: Stable}
\subsection{Target triangular array}\label{sub:target triangular array}
The first step is to describe a triangular array, consisting of finite valued random variables, which we will be able to embed in subsection \ref{sub: embedding} in every aperiodic, ergodic, probability preserving system.

Let $d_k:=\left[2^{\frac{2\alpha k}{2-\alpha}}\right]$. 

Consider the following triangular array of random variables: 
\begin{itemize}
	\item[(a)] For each $k\in\N$, $\left\{X_k(i):\ 0\leq i\leq 2d_k\right\}$ are i.i.d, $S\alpha S \left(k^{-1/\alpha}\right)$ random variables. 
	\item[(b)]   For each $k\in\N$, $\left\{X_k(i):\ 0\leq i\leq 2d_k\right\}$ is independent of\\ $\left\{X_j(i):\ 1\leq  j<k,\ 0\leq i\leq 2d_j \right\}$.
\end{itemize}

Let $(\Omega,\mathcal{F},\P)$ be a standard probability space on which all these random variables are defined. 

We now define a sequence of finite valued random variables as follows; First set 
\[
Y_k(j):=X_k(j)1_{\left[2^k\leq \left|X_k(j)\right|\leq 2^{k^2}\right] }. 
\]
Now let $\big(t_l^{(k)}\big)_{l=0}^{L_{k}}\subset \left[2^k,2^{k^2}\right]$, satisfying:
\begin{itemize}
	\item  $t_0^{(k)}=2^k, t_{L_k}^{(k)}=2^{k^2}$. Here $L_k+1\in\mathbb{N}$ is the number of points in the partition.
	\item For all $0\leq l< L_k$, $0<t_{l+1}^{(k)}-t_l^{(k)}\leq \frac{1}{d_k}$. 
\end{itemize}
Now for all $1\leq j\leq 2d_k$, let
\[
Z_{k}(j)=\begin{cases}
	\mathrm{sign}\left(Y_k(j)\right)t_l^{(k)},& \exists 0\leq l< L_{k},\ \ t_l^{(k)}\leq |Y_k(j)|<t_{l+1}^{(k)},\\
	0, & \left|Y_k(j)\right|\notin \left[2^k,2^{k^2}\right].
\end{cases}
\]
The following claim follows easily from the definition. 
\begin{fact}\label{fact: triangular}
The sequence $\left(Z_k(j)\right)_{\{k\in\N,1\leq j\leq 2 d_k \}}$ is a triangular array of random variables so that for every $k\in\N$, $\left(Z_{k}(j)\right)_{j=1}^{2 d_k}$ are finite-valued, i.i.d. random variables. 
\end{fact}
We summarise several key properties of the sequences defined above which will be used in the sequel.
\begin{lemma}\label{lem: facts on Z,X,Y}For every $k\in\N$, $1\leq j\leq 2d_k$:
	\begin{itemize}
		\item[(a)]  $\left|Z_k(j)-Y_k(j)\right|\leq \frac{1}{d_k}$.
		\item[(b)] $\mathbb{P}(Z_k(j)\neq 0)\leq \mathbb{P}\left(|X_k(j)|\geq 2^k\right)\leq\frac{C_\alpha}{k}2^{-\alpha k}$.  Here $C_\alpha$ is a global constant independent of $k$ and $j$. 
\end{itemize}
\end{lemma}
\begin{proof}
Part (a) and the first inequality in part (b) are immediate consequences of the definitions of $Y_k$ and $Z_k$ as functions of $X_k$. By Proposition \ref{prop: tails}, as $X_k(j)\sim^d \aS \big(\sqrt[\alpha]{1/k}\big)$ , there exists $C_\alpha$ (which is independent of $k$ and $j$) such that
\[
\mathbb{P}\left(|X_k(j)|\geq 2^k\right)\leq\frac{C_\alpha}{k}2^{-\alpha k}.
\]
\end{proof}
\subsection{Embedding the array in the dynamical system}\label{sub: embedding}
Let $(\mathcal{X},\B,m)$ be a standard probability space. A finite partition of $\mathcal{X}$ is \textbf{measurable} if all of its pieces (atoms) are Borel-measurable. 
Recall that a finite sequence of random variables $X_1,\ldots,X_n:\mathcal{X}\to \R$, each taking finitely many of values, is \textbf{independent of a finite partition} $\mathcal{P}=(P)_{P\in \mathcal{P}}$ if for all $s\in \R^n$ and $P\in\mathcal{P}$,
\[
m\left(\left(X_j\right)_{j=1}^n=s|P\right)=m\left(\left(X_j\right)_{j=1}^n=s\right).
\]
We will embed the triangular array using the following key proposition. 
\begin{proposition}\cite[Proposition 2]{MR4374685}\label{prop:KV20} Let $(\mathcal{X},\B,m,T)$ be an aperiodic, ergodic, probability preserving transformation and $\mathcal{P}$ a finite-measurable partition of $\mathcal{X}$. For every finite set $A$ and $U_1,U_2,\ldots,U_n$ an i.i.d. sequence of $A$ valued random variables, there exists $f:\mathcal{X}\to A$ such that $(f\circ T^j)_{j=0}^{n-1}$ is distributed as $(U_j)_{j=1}^n$ and $(f\circ T^j)_{j=0}^{n-1}$ is independent of $\mathcal{P}$. 
\end{proposition} 	
An easy corollary of this proposition and Fact \ref{fact: triangular} is the following. 
\begin{corollary}\label{cor: embed}
 Let $(\mathcal{X},\B,m,T)$ be an aperiodic, ergodic, probability preserving transformation and $\left(Z_k(j)\right)_{\{k\in\N,1\leq j\leq 2 d_k \}}$ be the triangular array from subsection \ref{sub:target triangular array}. There exist functions $f_k:\mathcal{X}\to\mathbb{R}$ such that $\left(f_k\circ T^{j-1}\right)_{\{k\in\N,1\leq j\leq 2d_k \}}$ is distributed as $\left(Z_k(j)\right)_{\{k\in\N,1\leq j\leq 2d_k\}}$.
\end{corollary}	
\begin{proof}
Starting with $\mathcal{P}=\{\mathcal{X}\}$, the trivial partition, and applying Proposition \ref{prop:KV20} we find $f_1:\mathcal{X} \to \R$ such that  $f_1,f_1\circ T,\dots,f\circ T^{2d_1-1}$ are i.i.d. distributed as the finite-valued random variable $Z_1(1)$.

 In the inductive step we are given  $f_k:\mathcal{X}\to\R$, $1\leq k\leq m$,  such that the array $\left\{f_k\circ T^{j-1}:  1\leq k\leq m,\ 1\leq j\leq 2d_k\right\}$ is distributed as
 $\left\{Z_k(j):1\leq k\leq m,\  1\leq j\leq 2 d_k\right\}$. 

Let $\mathcal{P}_m$ be the finite  partition of $\mathcal{X}$ according to the values of the (finite valued) random vector $V:=\left(f_k\circ T^{j-1}:  1\leq k\leq m,\ 1\leq j\leq 2d_k\right)$. 

Apply Proposition \ref{prop:KV20} and obtain a function $f_{m+1}:\mathcal{X}\to \R$ such that $\left\{f_{m+1}\circ T^{j-1}:\ 1\leq j\leq 2d_{m+1}\right\}$ is an i.i.d. sequence distributed as $\left\{Z_{m+1}(j):\ 1\leq j\leq 2d_{m+1}\right\}$ and independent of $\mathcal{P}_m$. Since being independent of $\mathcal{P}_m$ is equivalent to being independent of $V$, we see that 
 that the array $\left\{f_k\circ T^{j-1}:  1\leq k\leq m+1,\ 1\leq j\leq 2d_{k}\right\}$ is distributed as
$\left\{Z_k(j):1\leq k\leq m+1,\  1\leq j\leq 2d_k\right\}$. 
\end{proof}	
\subsection{Definition of the function} Let $(\mathcal{X},\B,m,T)$ be an aperiodic, ergodic, probability preserving system and $(f_k)_{k=1}^\infty$ the functions from Corollary \ref{cor: embed}.
\begin{lemma}
For $m$ almost every $\omega\in\mathcal{X}$, there exists $K(\omega)\in\N$ such that for all $k>K(\omega)$, $f_k(\omega)-f_k\circ T^{d_k}(\omega)=0$.   
\end{lemma}	
\begin{proof}
By the definition of the functions, we have for all $k\in\N$, 
\[
m\left(f_k\neq 0\right)=\mathbb{P}\left(Z_k(0)\neq 0\right).
\]
By Lemma \ref{lem: facts on Z,X,Y}.(b), there exists $C_\alpha$ such that for all $k$, 
\[
\mathbb{P}\left(Z_k(0)\neq 0\right)\leq \frac{C_\alpha}{k2^{k\alpha}}. 
\]
Consequently, as $T$ is $m$ preserving, 
\begin{align*}
\sum_{k=1}^\infty m\left(f_k-f_k\circ T^{d_k}\neq 0\right)&\leq 2\sum_{k=1}^\infty m\left(f_k\neq 0\right)\\
&\leq 2\sum_{k=1}^\infty \mathbb{P}\left(Z_k(0)\neq 0\right)<\infty. 
\end{align*}
The conclusion follows from the Borel-Cantelli lemma. 
 \end{proof}	
Set
\[
f:=\sum_{k=1}^\infty \left(f_k-f_k\circ T^{d_k}\right).
\]	
This function is well defined as it is almost surely a sum of finitely many values. The following  theorem implies Theorem \ref{thm: CLT}. In what follows,  $\log(x)$ denotes the logarithm of $x$ in base $2$ and $\ln(x)$ is the natural logarithm of $x$. 

\begin{manualtheorem}{1.b}
\label{thm: CLT2}
	$\frac{S_n(f)}{n^{1/\alpha}}\Rightarrow ^dS\alpha S \left(\sigma\right)$ with $\sigma^\alpha=2\ln\left(\frac{2}{2-\alpha}\right)$.
\end{manualtheorem}

\section{Proof of Theorem \ref{thm: CLT2}}\label{sec: proofs}

For a measurable function $g:\mathcal{X}\to\R$ and $k\in \N$, we write $U^kg=g\circ T^k$ . The proof of  Theorem \ref{thm: CLT2} begins by writing  
\begin{equation}\label{eq: decomp}
S_n(f)=S_n^{(\mathbf{S})}(f)+S_n^{(\mathbf{M})}(f)+S_n^{(\mathbf{L})}(f) 
\end{equation}
where 
\begin{align*}
S_n^{(\mathbf{M})}(f)&:= \sum_{k=(\frac{1}{\alpha}-\frac{1}{2})\log(n)+1}^{\frac{1}{\alpha}\log(n)}\left(S_n(f_k)-U^{d_k}S_n(f_k)\right)\\
S_n^{(\mathbf{S})}(f)&:=	\sum_{k=1}^{(\frac{1}{\alpha}-\frac{1}{2})\log(n)}\left(S_n(f_k)-U^{d_k}S_n(f_k)\right)\\
S_n^{(\mathbf{L})}(f)&:=\sum_{k=\frac{1}{\alpha}\log n+1}^\infty \left(S_n(f_k)-U^{d_k}S_n(f_k)\right).			
\end{align*}
Theorem \ref{thm: CLT2} follows from the following proposition,  \eqref{eq: decomp} and the converging together lemma (also known as Slutsky's Theorem). 
\begin{proposition}\label{MainCLT} $ $
	\begin{itemize}
		\item[(a)] $\frac{1}{n^{1/\alpha}}S_n^{(\mathbf{S})}(f)\to 0$ in probability. 
		\item[(b)] $\frac{1}{n^{1/\alpha}}S_n^{(\mathbf{M})}(f)\Rightarrow^d S_\alpha S \left(\sqrt[\alpha]{2\log\left(\frac{2}{2-\alpha}\right)}\right)$.
		\item[(c)] $\lim_{n\to\infty}m\left(S_n^{(\mathbf{L})}(f)\neq 0\right)=0$. 
\end{itemize}\end{proposition}
We first prove the simplest part. 
\begin{proof}[Proof of Propositiion \ref{MainCLT}.(c)]
By Corollary \ref{cor: embed} and Lemma \ref{lem: facts on Z,X,Y}.(b) for every  $k>\frac{1}{\alpha}\log(n)$, 
\[
m\left(f_k\neq 0\right)=\P\left(Z_k(1)\neq 0\right)\leq C_\alpha \frac{2^{-\alpha k}}{k}. 
\]
We have for all $k>\frac{1}{\alpha}\log(n)$,                                                                                                        
\begin{align*}
m\left(S_n(f_k)-U^{d_k}S_n(f_k)\neq 0\right)&\leq m\left(\exists j\in[0,n)\cup [d_k,d_k+n], f_k\circ T^j\neq 0\right)\\
&\leq \sum_{j=0}^{n-1} \left[m\left(f_k\circ T^j\neq 0\right)+m\left(f_k\circ T^{d_k+j}\neq 0\right)\right]\\
&\leq 2n\cdot m\left(f_k\neq 0\right)\leq 2nC_\alpha \frac{2^{-\alpha k}}{k}. 
\end{align*}
Here the third inequality is the union bound. A similar argument using the union bound gives
\begin{align*}
\P\left(S_n^{(\mathbf{L})}(f)\neq 0\right)&\leq \sum_{k=\frac{1}{\alpha}\log(n)+1}^\infty m\left(S_n(f_k)-U^{d_k}S_n(f_k)\neq 0\right)\\
&\leq 2nC_\alpha \sum_{k=\frac{1}{\alpha}\log(n)+1}^\infty \frac{2^{-k\alpha}}{k}\\
&\leq \frac{2nC_\alpha}{\frac{1}{\alpha}\log(n)}\sum_{k=\frac{1}{\alpha}\log(n)+1}^\infty 2^{-k\alpha}\lesssim \frac{1}{\log(n)}\xrightarrow[n\to\infty]{}0. 
\end{align*}
\end{proof}
\subsection{Proving Proposition \ref{MainCLT}.(b)}\label{sub: prop 7.(b)}
For $W\in \{X,Y,Z\}$, write 
\[
	S_n^{(\mathbf{M})}(\mathbf{W})= \sum_{k=(\frac{1}{\alpha}-\frac{1}{2})\log(n)+1}^{\frac{1}{\alpha}\log(n)}\sum_{j=1}^n\left(W_k(j)-W_k(j+d_k)\right) 
\]
Proposition \ref{MainCLT}.(b) follows from the following two Lemmas.
\begin{lemma}\label{lem: middle 1}
For all large $n$, $S_n^{(\mathbf{M})}(f)=^dS_n^{(\mathbf{M})}(\mathbf{Z})$.
\end{lemma}
\begin{lemma}\label{lem: middle 2}
$ $
\begin{itemize}
	\item[(a)]  $n^{-1/\alpha}\left|S_n^{(\mathbf{M})}(\mathbf{Z})-S_n^{(\mathbf{M})}(\mathbf{Y})\right|\xrightarrow[n\to\infty]{}0$.
	\item[(b)] $\frac{S_n^{(\mathbf{M})}(\mathbf{Y})}{n^{1/\alpha}}\Rightarrow^d \aS\left(\sqrt[\alpha]{2\ln\left(\frac{2}{2-\alpha}\right)}\right)$
\end{itemize}
\end{lemma} 
\begin{proof}[Proof of Proposition \ref{MainCLT}.(b)]
By Lemma \ref{lem: middle 1}, it suffices to show convergence of $\frac{S_n^{(\mathbf{M})}(\mathbf{Z})}{n^{1/\alpha}}$. To that end, write 
\[
\frac{S_n^{(\mathbf{M})}(\mathbf{Z})}{n^{1/\alpha}}=\frac{S_n^{(\mathbf{M})}(\mathbf{Z})-S_n^{(\mathbf{M})}(\mathbf{Y})}{n^{1/\alpha}}+\frac{S_n^{(\mathbf{M})}(\mathbf{Y})}{n^{1/\alpha}}.
\]
It follows from Lemma \ref{lem: middle 2} and the convergence together lemma that 
\[
 \frac{S_n^{(\mathbf{M})}(\mathbf{Z})}{n^{1/\alpha}}\Rightarrow^d \aS\left(\sqrt[\alpha]{2\ln\left(\frac{2}{2-\alpha}\right)}\right).
\]
\end{proof}	
\begin{proof}[Proof of Lemma \ref{lem: middle 1} and Lemma \ref{lem: middle 2}.(a)]
Note that  if $k\geq \left(\frac{1}{\alpha}-\frac{1}{2}\right)\log(n)$ then $d_k=\left[2^{\frac{2\alpha k}{2-\alpha}}\right]\geq n$. Consequently $n+d_k\leq 2d_k$ and  
\[
S_n^{(\mathbf{M})}(f)=^dG_n\left(f_k\circ T^{j-1}:\ k\in\N,\ 1\leq j\leq 2d_k \right) 
\]
where 
\[
G_n\left(x_{k,j}:\ k\in\N,\ 1\leq j\leq 2d_k \right)=\sum_{k=(\frac{1}{\alpha}-\frac{1}{2})\log(n)+1}^{\frac{1}{\alpha}\log(n)}\sum_{j=1}^n\left(x_{k,j}-x_{k,j+d_k}\right) 
\]
Since $G_n$ is continuous and $\left(f_k\circ T^{j-1}\right)_{\{k\in\N,1\leq j\leq 2d_k \}}$ is distributed as $\left(Z_k(j)\right)_{\{k\in\N,1\leq j\leq 2d_k \}}$, we see that for all large $n$,
\begin{align*}
S_n^{(\mathbf{M})}(f)&=^dG_n\left(Z_k(j):\ k\in\N,\ 1\leq j\leq 2d_k \right)\\
&=S_n^{(\mathbf{M})}(\mathbf{Z}),
\end{align*}
concluding the proof of Lemma \ref{lem: middle 1}. 

Now by Lemma \ref{lem: facts on Z,X,Y}.(a), if $k\geq \left(\frac{1}{\alpha}-\frac{1}{2}\right)\log(n)$, then 
\begin{align*}
\sum_{j=1}^n \left(\left|Z_k(j)-Y_k(j)\right|+\left|Z_k(j+d_k)-Y_k(j+d_k)\right|\right)\leq \frac{n}{d_k}\leq 1.
\end{align*}
Lemma \ref{lem: middle 2}.(a) readily follows from this as for all large $n$ 
\[
\left|S_n^{(\mathbf{M})}(\mathbf{Z})-S_n^{(\mathbf{M})}(\mathbf{Y}) \right|\leq \frac{1}{\alpha}\log(n).
\]
\end{proof}	
The proof of Lemma \ref{lem: middle 2}(b) is more involved and is done in two stages. The first stage, which is Lemma \ref{lem: 1 CLT}, is to interchange the $Y_k$ random variables with  $X_k$'s. The second, Lemma \ref{lem: 2 CLT}, is to show the  distributional convergence of $n^{-1/\alpha}S_n^{(\mathbf{M})}(\mathbf{X})$. 
 \begin{lemma}\label{lem: 1 CLT}
 $\frac{1}{n^{1/\alpha}}\left(S_n^{(\mathbf{M})}(\mathbf{Y})-S_n^{(\mathbf{M})}(\mathbf{X})\right)$ converges to $0$ in probability. 
 \end{lemma}
\begin{lemma}\label{lem: 2 CLT}$ $
$\frac{1}{n^{1/\alpha}}S_n^{(\mathbf{M})}(\mathbf{X})\Rightarrow ^d S_\alpha S\left(\sqrt[\alpha]{2\ln\left(\frac{2}{2-\alpha}\right)}\right)$. 
\end{lemma}

\begin{proof}[Proof of Lemma \ref{lem: 2 CLT}]
For every $k>\left(\frac{1}{\alpha}-\frac{1}{2}\right)\log(n)$, $d_k>n$ and $d_k+n\leq 2d_k$. 

Therefore, for all $n$ and $k>\left(\frac{1}{\alpha}-\frac{1}{2}\right)\log(n)$,  $\left(X_k(j)-X_k(j+d_k)\right)_{k\in\N,1\leq j\leq n}$ is a sequence of i.i.d. $\aS\left(\sqrt[\alpha]{\frac{2}{k}}\right)$ random variables. It follows that 
\[
\sum_{j=1}^n \left(X_k(j)-X_k(j+d_k)\right)\sim^d \aS\left(\left(\frac{2n}{k}\right)^{1/\alpha}\right).
\]

Secondly, since $\left\{X_k(j):k\in\N,\ 1\leq j\leq 2d_k\right\}$ are independent, we see that 
\[
\left\{\sum_{j=1}^n \left(X_k(j)-X_k(j+d_k)\right):\ k>\left(\frac{1}{\alpha}-\frac{1}{2}\right)\log(n)\right\}
\] 
are independent $\aS$ random variables. 
As a result, $\frac{1}{n^{1/\alpha}}S_n^{(\mathbf{M})}(\mathbf{X})$ is $S_\alpha S(\sigma_n)$ distributed with 
\[
\sigma_n^\alpha=\frac{1}{n}\sum_{k=\left(\frac{1}{\alpha}-\frac{1}{2}\right)\log(n)+1}^{\frac{1}{\alpha}\log(n)}\frac{2n}{k}\sim 2\ln\left(\frac{2}{2-\alpha}\right),\ \text{as}\ n\to\infty. 
\]	
We conclude from this and Fact \ref{Facts of life} that 
\[
\frac{1}{n^{1/\alpha}}S_n^{(\mathbf{M})}(\mathbf{X})\Rightarrow^d\aS \left(\sqrt[\alpha]{2\ln\left(\frac{2}{2-\alpha}\right)}\right).
\]

\end{proof}
The rest of this subsection is concerned with the proof of Lemma \ref{lem: 1 CLT}. Observe that 
\begin{equation}\label{eq: V and S}
S_n^{(\mathbf{M})}(\mathbf{X})-S_n^{(\mathbf{M})}(\mathbf{Y})=\overline{V}_n^{(\mathbf{M})}+\underbar{V}_n^{(\mathbf{M})}
\end{equation}
where
 \begin{align*}
 \overline{V}_n^{(\mathbf{M})}&= \sum_{k=(\frac{1}{\alpha}-\frac{1}{2})\log(n)+1}^{\frac{1}{\alpha}\log(n)}\sum_{j=1}^n \left(X_k(j)1_{\left[\left|X_k(j)\right|\geq 2^{k^2}\right]}- X_k(j+d_k)1_{\left[\left|X_k(j+d_k)\right|\geq 2^{k^2}\right] }\right) \\
  \underbar{V}_n^{(\mathbf{M})}&= \sum_{k=(\frac{1}{\alpha}-\frac{1}{2})\log(n)+1}^{\frac{1}{\alpha}\log(n)}\sum_{j=1}^n \left(X_k(j)1_{\left[\left|X_k(j)\right|\leq 2^{k}\right]}- X_k(j+d_k)1_{\left[\left|X_k(j+d_k)\right|\leq 2^{k}\right] }\right).\\
 \end{align*}
 \begin{lemma}\label{lem: V overline}
  $\frac{1}{n^{1/\alpha}}\overline{V}_n^{(\mathbf{M})}\to 0$ in probability. 
 \end{lemma}
\begin{proof}
We write for all $k,j\in\N$, $\widehat{X}_k(j)=X_k1_{\left[\left|X_k(j)\right|\geq 2^{k^2}\right] }$ so that for every $n\in\N$,
\[
\overline{V}_n^{(\mathbf{M})}=\sum_{k=\left(\frac{1}{\alpha}-\frac{1}{2}\right)\log(n)+1}^{\frac{1}{\alpha}\log(n)}\sum_{j=1}^n \left( \widehat{X}_k(j)-\widehat{X}_k(j+d_k)\right). 
\] For $k\in\mathbb{N}$, let $A_k$ be the event
\[
\left\{\exists j\in [1,2d_k], \widehat{X}_k(j)\neq 0 \right\}. 
\]
Similarly to the proof of Proposition \ref{MainCLT}.(c), there exists $C_\alpha>0$ such that for all but finitely many $k\in\mathbb{N}$, 
\begin{align*}
\P\big(A_k\big)&\leq 2d_k\P\left(\left|X_k(1)\right|\geq 2^{k^2}\right)\\
&\leq 2C_\alpha d_k2^{-\alpha k^2}. 
\end{align*}
The right hand side being summable,  the Borel-Cantelli lemma implies that $\P$- almost surely, $A_k$ happens only for finitely many $k$'s. We now deduce the claim from this fact. 

For all $ \left(\frac{1}{\alpha}-\frac{1}{2}\right)\log(n)\leq k\leq \frac{1}{\alpha}\log(n)$, $n\leq d_k$ and 
\[
\left[\sum_{j=1}^n \left( \widehat{X}_k(j)-\widehat{X}_k(j+d_k)\right)\neq 0\right]\subset A_k.
\]
 
Since $\log(n)\to\infty$ and almost surely $A_k$ happens finitely often we have $\lim_{n\to\infty}\frac{1}{n^{1/\alpha}}\overline{V}_n^{(\mathbf{M})} =0$ almost surely.



\end{proof}
\begin{lemma}\label{lem: V underline}
 $\frac{1}{n^{1/\alpha}}\underbar{V}_n^{(\mathbf{M})}\to 0$ in probability. 
\end{lemma}
For the proof of Lemma \ref{lem: V underline} we need the following variance bound.

\begin{proof}[Proof of Lemma \ref{lem: V underline}]
Write $\widetilde{X}_k(j)=X_k(j)1_{|X_k(j)|\leq 2^k}$. Fix $k>\left(\frac{1}{\alpha}-\frac{1}{2}\right)\log(n)$ so that $n\leq d_k$. The sequence $\left\{\tilde{X}_k(j):\ 1\leq j\leq n+d_k\right\}$ is an i.i.d sequence of symmetric random variables. We have,
\begin{align*}
\mathbb{E}\left(\left(\sum_{j=1}^n\left(\widetilde{X_k}(j)-\widetilde{X_k}(j+d_k)\right)\right)^2\right)&=\sum_{j=1}^n\left(\mathrm{Var}\left(\widetilde{X_k}(j)\right)+\mathrm{Var}\left(\widetilde{X_k}(j+d_k)\right)\right)\\
&=2n\mathrm{Var}\left(\widetilde{X}_k(1)\right). 
\end{align*}
Since $X_k(1)$ is $\aS\left(k^{-1/\alpha}\right)$ distributed, it follows from Lemma \ref{lem: variance of truncation} with $K=2^k$ that,
\[
\mathbb{E}\left(\left(\sum_{j=1}^n\left(\widetilde{X_k}(j)-\widetilde{X_k}(j+d_k)\right)\right)^2\right)\leq 2Cn\frac{2^{(2-\alpha)k}}{k}.
\]
Now by properties (a) and (b) of the array $\left(X_k(j)\right)_{k,j\in\N}$, $$\left\{\sum_{j=1}^n\left(\widetilde{X_k}(j)-\widetilde{X_k}(j+d_k)\right):\ \left(\frac{1}{\alpha}-\frac{1}{2}\right)\log(n)< k\leq \frac{1}{\alpha}\log(n)\right\}$$ are independent, zero mean random variables, therefore
\begin{align*}
\mathbb{E}\left(\left(n^{-1/\alpha}\underline{V}_n^{(\mathbf{M})}\right)^2\right)&=n^{-2/\alpha}\sum_{k=\left(\frac{1}{\alpha}-\frac{1}{2}\right)\log(n)+1}^{\frac{1}{\alpha}\log(n)}\mathbb{E}\left(\left(\sum_{j=1}^n\left(\widetilde{X_k}(j)-\widetilde{X_k}(j+d_k)\right)\right)^2\right)\\
&\leq 2Cn^{-2/\alpha} \sum_{k=\left(\frac{1}{\alpha}-\frac{1}{2}\right)\log(n)+1}^{\frac{1}{\alpha}\log(n)}\frac{n2^{(2-\alpha)k}}{k}\\
&\lesssim n^{-\frac{2}{\alpha}+1}\frac{2^{\left(\frac{2-\alpha}{\alpha}\right)\log(n)}}{\log(n)}=\frac{1}{\log(n)}\xrightarrow[n\to\infty]{}0. 
\end{align*}
A routine application of Markov's inequality shows that $n^{-1/\alpha}\underline{V}_n^{(\mathbf{M})}$ tends to $0$ in probability. 

\end{proof}	
\begin{proof}[Proof of Lemma \ref{lem: 1 CLT}]
The result readily follows from Lemmas \ref{lem: V overline} and \ref{lem: V underline} and equation \eqref{eq: V and S}. 

\end{proof}		
We have now concluded the proof of Proposition \ref{MainCLT}.(b). 
\subsection{Proving Proposition \ref{MainCLT}.(a)}
Write 
\[
\mathsf{G}_n(f)=\sum_{k=1}^{(\frac{1}{\alpha}-\frac{1}{2})\log(n)} S_{d_k}(f_k)
\]
\begin{proposition}\label{prop:small 1}$ $
	\begin{itemize}
		\item[(a)] For all $n\in\N$, $S_n^{(\mathbf{S})}(f)=\mathsf{G}_n(f)-U^n\left(\mathsf{G}_n(f)\right)$.
		\item[(b)] $\frac{1}{n^{1/\alpha}}\mathsf{G}_n(f)\to 0$ in probability.  
	\end{itemize} 
\end{proposition}
\begin{proof}[Proof of Proposition \ref{prop:small 1}.(a)]
	For all $k\leq \left(\frac{1}{\alpha}-\frac{1}{2}\right)\log(n)$, $d_k\leq n$. Consequently,
	\[
	S_n(f_k)-U^{d_k}S_n(f_k)=S_{d_k}(f_k)-U^{n}S_{d_k}(f_k).
	\]
	Identity (a) follows from summing these identities over all $1\leq k\leq \left(\frac{1}{\alpha}-\frac{1}{2}\right)\log(n)$. 
\end{proof}
The proof of part(b) in \ref{prop:small 1} is longer and goes along identical lines as in Subsection \ref{sub: prop 7.(b)}. Recall the notation 
\[\widehat{X}_k(j)=X_k(j)1_{\left[|X_k(j)|>2^{k^2}\right]}\ \text{and} \ \ 	\widetilde{X}_k(j)=X_k(j)1_{\left[|X_k(j)|<2^{k}\right]}.
\]
For $W\in \{X,\widetilde{X},\widehat{X},Y,Z\}$, write 
\[
\mathsf{G}_n(\mathbf{W})= \sum_{k=1}^{(\frac{1}{\alpha}-\frac{1}{2})\log(n)} \sum_{j=1}^{d_k} W_k(j) 
\]
Proposition \ref{prop:small 1}.(b) follows directly from the following lemma.
\begin{lemma}$ $
	\begin{itemize}
		\item[(a)] For every $n\in\N$, $\mathsf{G}_n(f)=^d\mathsf{G}_n(\mathbf{Z})$. 
		\vspace{0.5mm}
		\item[(b)] $\frac{\mathsf{G}_n(\mathbf{Z})-\mathsf{G}_n(\mathbf{Y})}{n^{1/\alpha}}\xrightarrow[n\to\infty]{}0$ pointwise. 
		\vspace{0.5mm}
		\item[(c)] $\frac{\mathsf{G}_n(\mathbf{Y})-\mathsf{G}_n(\mathbf{X})}{n^{1/\alpha}}\xrightarrow[n\to\infty]{}0$ in probability. 
		\vspace{0.5mm}
		\item[(d)] $\frac{\mathsf{G}_n(\mathbf{X})}{n^{1/\alpha}}\xrightarrow[n\to\infty]{}0$ in probability. 
		
	\end{itemize}
	
\end{lemma}

\begin{proof}
	Fix $n\in\N$ and note that $\mathsf{G}_n(f)$ is a continuous function of $\mathbf{F}_n:=\left(f_k\circ T^j:\ 1\leq k\leq \left(\frac{1}{\alpha}-\frac{1}{2}\right)\log(n), 0\leq j<d_k\right)$. Since $\mathbf{F}_n$ and $\left(Z_k(j):\ 1\leq k\leq \left(\frac{1}{\alpha}-\frac{1}{2}\right)\log(n), 1\leq j\leq d_k\right)$ are equally distributed we see that part (a) holds. 
	
	Similarly as in the proof of Lemma \ref{lem: middle 2}.(a), for all $n\in\N$, 
	\begin{align*}
		\left|\mathsf{G}_n(\mathbf{Z})-\mathsf{G}_n(\mathbf{Y})\right|&\leq \sum_{k=1}^{(\frac{1}{\alpha}-\frac{1}{2})\log(n)} \sum_{j=1}^{d_k}\left|Z_k(j)-Y_k(j)\right|\\
		&\leq \sum_{k=1}^{(\frac{1}{\alpha}-\frac{1}{2})\log(n)} \sum_{j=1}^{d_k} \frac{1}{d_k}=o\left(n^{1/\alpha}\right),
	\end{align*}
	concluding the proof of part (b). 
	
	Now for all $n\in\N$, 
	\[
	\mathsf{G}_n(\mathbf{X})-\mathsf{G}_n(\mathbf{Y})=\mathsf{G}_n\big(\mathbf{\widehat{X}}\big)-\mathsf{G}_n\big(\mathbf{\widetilde{X}}\big).
	\]
	Part (c) follows from Lemma \ref{lem: last}. 
	
	As in the proof of Lemma \ref{lem: 2 CLT}, $\sum_{j=1}^{d_k} X_k(j)\sim^d \aS\left(\sqrt[\alpha]{\frac{d_k}{k}}\right)$ as sum of $d_k$ i.i.d. $\aS\left(\sqrt[\alpha]{\frac{1}{k}}\right)$ random variables.  By the triangular array property, $\left\{\sum_{j=1}^{d_k} X_k(j):\ k\in\N\right\}$ are independent and consequently,
	\[
	\frac{\mathsf{G}_n(\mathbf{X})}{n^{1/\alpha}}=\frac{1}{n^{1/\alpha}} \sum_{k=1}^{(\frac{1}{\alpha}-\frac{1}{2})\log(n)}\left( \sum_{j=1}^{d_k} X_k(j) \right)\sim^d \aS\left(\sigma(n)\right),
	\]
	where \footnote{recall that for $k\leq (\frac{1}{\alpha}-\frac{1}{2})\log(n)$, $d_k\leq n$.}
	\[
	\sigma(n)^\alpha=\frac{1}{n}\sum_{k=1}^{(\frac{1}{\alpha}-\frac{1}{2})\log(n)}\frac{d_k}{k} \lesssim   \frac{1}{\log(n)}\xrightarrow[n\to\infty]{} 0. 
	\]
	Part (d) now follows from Fact \ref{Facts of life}. 
\end{proof}

\begin{lemma}
	\label{lem: last} $ $
	\begin{itemize}
		\item Almost surely, $\lim_{n\to\infty}\mathsf{G}_n\big(\mathbf{\widehat{X}}\big)=\sum_{k=1}^\infty \sum_{j=1}^{d_k}\widehat{X}_k(j)\in\R$.
		\item $n^{-1/\alpha}\mathsf{G}_n\big(\mathbf{\widetilde{X}}\big)\xrightarrow[n\to\infty]{}0$ in probability.
	\end{itemize}
	
\end{lemma}
\begin{proof}
	The proof of the first claim goes along similar lines to the proof of Lemma \ref{lem: V overline}. Write $A_k:=\left\{\exists j\in [1,d_k], \widehat{X}_k(j)\neq 0 \right\}$. By the union bound and Proposition \ref{prop: tails}, 
	\[
	\P\left(A_k\right)\leq d_k\P\left(\left|X_k(1)\right|>2^{k^2}\right)\leq C_\alpha \frac{d_k}{k}2^{-\alpha k^2}.
	\] 	
	Since the right hand side is summable, it follows from the Borel-Cantelli lemma that almost surely, $A_k$ holds for only finitely many $k$. This implies that almost surely 
	\[
	\#\left\{(k,j)\in\N^2: X_k(j)\neq 0\right\}<\infty.
	\]
	Consequently $\sum_{k=1}^\infty \sum_{k=1}^{d_k}\widehat{X}_k(j)$ is almost surely a sum of finitely many terms. This concludes the proof of the first part. 
	
	For the second part, note that by independence of $\left(X_k(j)\right)_{j=1}^{d_k}$ and Lemma \ref{lem: variance of truncation}, there exists $C>0$ so that
	\[
	\mathrm{Var}\left(\sum_{j=1}^{d_k} \widetilde{X}_k(j)\right)=\sum_{j=1}^{d_k}\mathrm{Var}\left(\widetilde{X}_k(j)\right)\leq C \frac{d_k2^{(2-\alpha) k}}{k}.
	\]
	As $\left\{\sum_{j=1}^{d_k}\widetilde{X}_k(j):\ k\in\N\right\}$ are independent, centred and square integrable random variables, writing $\kappa_n=\left(\frac{1}{\alpha}-\frac{1}{2}\right)\log(n)$, we have
	
	\begin{align*}
		\mathbb{E}\left(\left(n^{-1/\alpha}\mathsf{G}_n\big(\mathbf{\widetilde{X}}\big) \right)^2\right)&=n^{-2/\alpha}\sum_{k=1}^{\kappa_n}\mathrm{Var}\left(\sum_{j=1}^{d_k} \widetilde{X}_k(j)\right)\\
		&\leq Cn^{-2/\alpha}\sum_{k=1}^{\kappa_n} \frac{d_k2^{(2-\alpha) k}}{k}\\
		&\lesssim n^{-2/\alpha}d_{\kappa_n}\frac{2^{(2-\alpha) \kappa_n}}{\kappa_n}\\
		&\leq n^{1-\frac{2}{\alpha}}2^{\frac{(2-\alpha)^2}{2\alpha}\log(n)}=n^{1-\frac{2}{\alpha}+\frac{(2-\alpha)^2}{2\alpha}}\xrightarrow[n\to\infty]{}0,
	\end{align*}
	since for $\alpha\in(0,2)$, 
	\[
	1-\frac{2}{\alpha}+\frac{(2-\alpha)^2}{2\alpha}=\frac{\alpha^2-2\alpha}{2\alpha}<0. 
	\]
	The second part follows from a routine application of Markov's inequality. 
\end{proof}
We can now conclude the proof of Proposition \ref{MainCLT}.(a). 
\begin{proof}[Proof of Proposition \ref{MainCLT}.(a)]
	Since $\mathsf{G}_n(f)=^dU^n\left(\mathsf{G}_n(f)\right)$,  it follows from Proposition \ref{prop:small 1}.(b) that $\frac{1}{n^{1/\alpha}}\mathsf{G}_n(f)$ and $\frac{1}{n^{1/\alpha}}U^n\left(\mathsf{G}_n(f)\right)$ converge to $0$ in probability. By Proposition \ref{prop:small 1} we see that $\frac{1}{n^{1/\alpha}}S_n^{(\mathbf{S})}(f)$ converges to $0$ in probability. 
\end{proof}

\section{Appendix: Growth of dispersion for stationary $\aS$ processes}
As $\aS$ random variables are defined by their characteristic functions, L\'evy's continuity theorem implies the following fact. 
\begin{fact}\label{Facts of life}
If for all $n\in\N$, $Z_n$ is $\aS(\sigma_n)$ distributed and $\lim_{n\to\infty}\sigma_n^\alpha=A^\alpha$, then $Z_n\Rightarrow^d \aS(A)$. In addition,
If $\sigma_n\to 0$ then $Z_n\Rightarrow^d 0$. 
\end{fact}
The following tail bound is used extensively in this work. 
\begin{proposition}\label{prop: tails}
	There exists $C_\alpha>0$ such that for all $0<\sigma\leq 1$, if $X$ is an $\aS(\sigma)$ random variable and $t\geq 1$ then,
	\[
	\P\left(|X|\geq t\right)\leq  C_\alpha \sigma^\alpha  t^{-\alpha}
	\]
\end{proposition}
\begin{proof}
By Proposition 1.2.15 in \cite{MR1280932}, there exists $c_\alpha>0$ such that if $X\sim^d \aS(1)$, then
\[
	\P\left(|X|\geq t\right)\sim  c_\alpha   t^{-\alpha},\ \ \text{as}\ t\to\infty. 
\]
We deduce that 
\[
C_\alpha:= \sup_{t\geq 1}\frac{\P\left(|X|\geq t\right)}{t^{-\alpha}}<\infty. 
\]
Finally if $X\sim^d\aS(\sigma)$ with $\sigma\leq 1$ and $t\geq 1$, we have 
\[
\P\left(|X|\geq t\right)=\P\left(\left|\frac{X}{\sigma}\right|\geq \frac{t}{\sigma}\right)\leq C_\alpha \left(\frac{t}{\sigma}\right)^{-\alpha}. 
\]
\end{proof}
The tail bound implies the following inequality for the variance.
\begin{lemma}\label{lem: variance of truncation}
	There exists $c=c(\alpha)>0$ such that for all $K\geq 1$, $0<\sigma\leq1$ and $m\in\N$, if $X$ is a $S\alpha S(\sigma)$ random variable, then
	\[
	\mathrm{Var}\left(X1_{\left[|X|\leq K\right]} \right)\leq cK^{2-\alpha}\sigma^{\alpha}.
	\]
\end{lemma}

\begin{proof}
	As $X$ is symmetric the random variable $X1_{\left[|X|\leq K\right]}$ has zero mean.	By Proposition \ref{prop: tails} there exists $C_\alpha>0$ such that, 
	\begin{align*}
		\mathrm{Var}\left(X 1_{\left[|X|\leq K\right]} \right)&=\mathbb{E}((X1_{[|X|\leq K]})^2)\\
		&=\int x\P(|X|1_{|X|\geq K}>x)dx\\
		&\leq 1+\int_1^K x\P(|X|>x)dx\\
		&\leq [1+o_{K\to\infty}(1)]C_\alpha\sigma^\alpha\int_1^K x^{1-\alpha}dx \\
		&=C_\alpha\cdot\sigma^\alpha \frac{K^{2-\alpha}}{2-\alpha}[1+o_{K\to\infty}(1)].
	\end{align*}
	We conclude that there exists $C$ depending only on $\alpha$ such that for all $K\geq 1$, 
	\[
	\mathrm{Var}\left(X1_{[|X|\leq K]}\right)\leq C\sigma^\alpha K^{2-\alpha}. 
	\]
\end{proof}

In our construction of CLT functions we used a triangular array of random variables $Y_k$ which are not $\aS$ distributed but are in the domain of attraction of an $\aS$ distribution. A main reason for this choice lies in the fact that the dispersion of a stationary $\aS$ process does not go fast enough for the methods of \cite{MR1624218} to work.

 A real valued stationary process $(X_n)_{n=1}^\infty$ is a $\aS$ process if every $Z$ in the linear span of $\{X_n:n\in \N\}$ is $\aS$ distributed. In that case the function 
\[
\|X\|_\alpha=\left(-\log\mathbb{E}\left(e^{iX}\right)\right)^{1/\alpha},
\]
is a quasi-norm from $\mathrm{Lin}(\mathbf{X}):=\mathrm{span}\{X_n:n\in\N\}$ to $[0,\infty)$ and for all $Z\in\mathrm{Lin}(\mathbf{X})$, $\|Z\|_\alpha$ equals the dispersion parameter of $Z$. The following is a well known fact on stationary $\aS$ processes. 
\begin{fact}\label{Fact: Lesigne}
If $0<\alpha<1$ and $(X_n)_{n=1}^\infty$ is a stationary $\aS$ process, then for every $N\in\N$, 
\[
\left\|\sum_{j=1}^N X_j\right\|_\alpha^\alpha\leq N\left\|X_1\right\|_\alpha^\alpha.
\]
\end{fact}
\begin{proof}
By \cite[Property 2.10.5]{MR1280932}, if $X,Y$ are $\aS$ random variables with $0<\alpha<1$, then
\[
\|X\|_\alpha^\alpha+\|Y\|_\alpha^\alpha-\left\|X+Y\right\|_\alpha^\alpha\geq 0. 
\]
A straightforward inductive procedure gives the claim. 
\end{proof}
\begin{remark}
One can show using stochastic integrals that there is equality if and only if $X_1,\ldots,X_N$ are independent. 
\end{remark}

\subsection*{Acknowledgement}
We thank the referee for his/her valuable remarks.

\bibliographystyle{abbrv}
\bibliography{embedding.bib}

\end{document}